\documentclass[12pt]{amsart}

\usepackage[T1]{fontenc}
\usepackage[utf8]{inputenc}
\usepackage{lmodern}
\usepackage{microtype}
\usepackage{mathtools}
\usepackage{amssymb}
\usepackage{enumitem}
\usepackage[colorlinks=true,linkcolor=blue,citecolor=blue,urlcolor=blue]{hyperref}

\newcommand{\Triv}{\operatorname{Triv}}
\newcommand{\Aut}{\operatorname{Aut}}

\newtheorem{athm}{Theorem}

\theoremstyle{plain}
\newtheorem{theorem}{Theorem}[section]
\newtheorem{lemma}[theorem]{Lemma}
\newtheorem{corollary}[theorem]{Corollary}
\newtheorem{Proposition}[theorem]{Proposition}
\newtheorem{example}{Example}
\theoremstyle{definition}

\newtheorem*{corollary*}{Corollary}
\theoremstyle{remark}
\newtheorem{remark}[theorem]{Remark}

\title[Ideals and Solvability in Skew Braces]
{Ideals and Solvability in Skew Braces}

\author{Marco Damele}
\address{Department of Mathematics and Computer Science,
University of Cagliari, Via Ospedale 72, 09124 Cagliari, Italy}
\email{marco.damele@unica.it}

\author{G\"ulİn Ercan}
\address{Department of Mathematics, Middle East Technical University, Ankara, Turkey}
\email{ercan@metu.edu.tr}

\subjclass[2020]{}
\keywords{}

\begin{document}

\begin{abstract}
We investigate how nilpotency assumptions on the multiplicative group
of a finite skew brace constrain its ideal structure and solvability.
Our first main result shows that, if \(B=(B,+,\cdot)\) is finite and
\((B,\cdot)\) is nilpotent, then the additive Fitting subgroup
\(F(B,+)\) is a non-zero ideal of \(B\). As consequences, every finite
simple skew brace with nilpotent multiplicative group is isomorphic to
\(\Triv(C_p)\) for some prime \(p\), and every such skew brace admits an
ideal of prime index. In particular,
\[
B*B\neq B
\qquad\text{and}\qquad
\partial(B) \neq B.
\]
We also show that nilpotency of the multiplicative group does not, in
general, imply either left nilpotency or solvability.

Motivated by this obstruction, we then study the solvability of
two-sided skew braces, both in the finite and in the general setting.
We prove that, whenever \(I\) is an ideal of a two-sided skew brace
\(B\), the internal commutator ideal \([I,I]_I\) is again an ideal of
\(B\). This yields an extension theorem for solvability and implies
that, for finite two-sided skew braces, solvability of the skew brace is
equivalent to solvability of either the additive or the multiplicative
group. In particular, every finite skew brace with abelian
multiplicative group is solvable. Finally, we show that, although this
equivalence fails in general for infinite two-sided skew braces, a
residual form of it still survives: if the additive group is solvable,
then every finite homomorphic image of the multiplicative group is
solvable.
\end{abstract}

\maketitle

\noindent\textbf{2020 Mathematics Subject Classification.}
Primary 16T25; Secondary 20D10, 20D15, 20F16.

\medskip

\noindent\textbf{Keywords.}
Skew brace, ideal, solvable skew brace, two-sided skew brace,
nilpotent multiplicative group, Fitting subgroup, finite quotient.

\section{Introduction}

Skew braces were introduced by Guarnieri and Vendramin
\cite{Guarnieri_2016} as a non-commutative generalization of
the braces introduced by Rump in \cite{Rump2007}. A skew brace is a
triple $B=(B,+,\cdot)$
such that \((B,+)\) and \((B,\cdot)\) are groups on the same underlying
set and
\[
a\cdot(b+c)=a\cdot b-a+a\cdot c
\]
for all \(a,b,c\in B\).
When the additive group is abelian, one recovers the original notion
of a brace.

Skew braces arise naturally in the study of non-degenerate
set-theoretic solutions of the Yang--Baxter equation. More precisely,
they provide an algebraic framework in which such solutions can be
constructed and investigated, and they allow structural properties of
the solutions to be related to properties of the associated skew
braces; see, for instance,
\cite{Guarnieri_2016,
CedoSmoktunowiczVendramin2019,
BallesterBolinchesEstebanRomeroJimenezSeralPerezCalabuig2024,CV}.
They are also closely related to regular subgroups of holomorphs and,
consequently, to Hopf--Galois structures; see, among others,
\cite{ByottSolubility2015,TsangQinSolvability,ByottInsolubleTransitive}.

A natural and interesting problem in the theory is to understand how
the internal structure of a skew brace is influenced by the
group-theoretic properties of its additive and multiplicative groups. In particular,
one may ask to what extent properties of one of the groups force
analogous properties of the other group or impose restrictions on the
ideals of the skew brace.

To the best of the authors' knowledge, only a few results are currently
available in this direction. Structural restrictions for two-sided skew
braces were obtained in
\cite{TrappeniersTwoSided,DamelePerfectTwoSided,DameleZGroups}.
Further results under cyclicity assumptions on Sylow subgroups were
established in \cite{DameleSimpleCyclicSylow}, including a classification
of finite simple skew braces with multiplicative \(Z\)-group and several
splitting theorems.

Among the most important structural properties of a skew brace are
simplicity and solvability. The solvability of skew braces and its
relationship with the solvability of the associated Yang--Baxter
solutions have been studied in
\cite{BallesterBolinchesEstebanRomeroJimenezSeralPerezCalabuig2024}.
Simple braces and simple skew braces have also attracted considerable
attention. Important constructions and structural results for simple
braces were obtained by Bachiller and by Cedó, Jespers and Okniński;
see
\cite{BachillerSimpleBraces,BachillerCedoJespersOkninski,
CedoJespersOkninski}.
For simple skew braces, see in particular
\cite{SmoktunowiczVendramin2018,ByottSimpleSkewBraces}.
The behaviour of simple skew braces can differ substantially from that
of classical braces and finite groups. In particular, simple skew
braces of odd order may occur, showing that the interaction between
the two group operations can produce phenomena which are not visible
from either associated group separately.

The present paper continues this line of research by investigating how
nilpotency assumptions on one of the associated groups constrain the
structure of the skew brace. Finite skew braces of nilpotent type,
namely skew braces whose additive group is nilpotent, were studied
systematically by Cedó, Smoktunowicz and Vendramin in
\cite{CedoSmoktunowiczVendramin2019}.

In the first part of this paper, we consider the dual
situation in which the multiplicative group is nilpotent, and we
investigate how this hypothesis affects the ideal structure of the
skew brace.

Our first main result shows that the additive Fitting subgroup is
always an ideal.

\begin{athm}\label{intro:theorem-A}
Let \(B=(B,+,\cdot)\) be a non-zero finite skew brace such that
\((B,\cdot)\) is nilpotent. Then \(F(B,+)\) is a non-zero ideal of
\(B\).
\end{athm}

The proof combines the existence of Hall subbraces with a
group-theoretic fixed-point argument. More precisely, for every prime
\(p\) dividing \(|F(B,+)|\), we prove that \(O_p(B,+)\) is an ideal of
\(B\). The main difficulty is to establish its normality in the
multiplicative group. This is achieved by studying the action induced
by its elements on $(B,+)/O_p(B,+)$
and applying a fixed-point lemma for \(p\)-power order automorphisms of
finite solvable groups. As an immediate consequence, finite simple skew braces with nilpotent
multiplicative group are completely determined.

\begin{corollary*}
Let \(B\) be a finite simple skew brace such that \((B,\cdot)\) is
nilpotent. Then $B\simeq \Triv(C_p)$
for some prime \(p\).
\end{corollary*}

Here, for a group \(G\), we denote by \(\Triv(G)\) the trivial skew brace over
\(G\), namely the skew brace whose additive and multiplicative groups
both coincide with \(G\) and whose two operations are identical.

This result may be viewed as the finite analogue of
\cite{DameleLoiSimpleLieNilpotent}. There it is proved that a simply
connected simple Lie skew brace with nilpotent multiplicative Lie group
is necessarily one-dimensional and abelian, and hence isomorphic to
the trivial Lie skew brace \((\mathbb{R},+,+)\).

We also obtain the following consequences.

\begin{corollary*}\label{intro:prime-index}
Let \(B\) be a finite skew brace such that \((B,\cdot)\) is nilpotent.
Then \(B\) admits an ideal of index \(p\) for some prime \(p\).
\end{corollary*}

\begin{corollary*}\label{intro:non-perfect}
Let \(B\) be a finite skew brace such that \((B,\cdot)\) is nilpotent.
Then $B*B\neq B$ and $\partial(B)\neq B.$
\end{corollary*}

On the other hand, Example~\ref{example:nilpotent-not-left-nilpotent}
shows that the nilpotency of the multiplicative group does not force
the skew brace to be left nilpotent and does not guarantee the
existence of an ideal of index \(p\) for every prime divisor \(p\) of
its order.

In view of Corollary \ref{cor: simple-multiplicative-nilpotent} it is natural to ask whether a finite skew brace with nilpotent
multiplicative group must be solvable. The answer is negative. Indeed,
the skew brace $\operatorname{SmallBrace}(32,24003)$
from
\cite[Example~38]{BallesterBolinchesEstebanRomeroJimenezSeralPerezCalabuig2024}
is not solvable. Since it has order \(32\), its multiplicative group is
a \(2\)-group and hence is nilpotent. Thus nilpotency of the
multiplicative group alone is too weak to imply solvability.

This motivates the stronger question of whether a finite skew brace
with abelian multiplicative group must be solvable. Recall that a skew
brace \(B=(B,+,\cdot)\) is called \emph{two-sided} if, in addition to
the usual skew brace identity, it also satisfies
\[
(a+b)\cdot c=a\cdot c-c+b\cdot c
\]
for all \(a,b,c\in B\). Every skew brace with abelian multiplicative
group is two-sided. Therefore, in the second part of the paper, we
study the solvability of finite two-sided skew braces. The key result
is the following ideal-theoretic property.

\begin{athm}\label{intro:theorem-B}
Let \(B\) be a two-sided skew brace and let \(I\trianglelefteq B\).
Then the first term of the derived series of \(I\), computed in \(I\),
is an ideal of \(B\); that is,
\[
\partial_I^1(I)=[I,I]_I\trianglelefteq B.
\]
\end{athm}

In an arbitrary skew brace, an ideal which is solvable as a skew brace
in its own right need not be solvable with respect to the ambient skew
brace. Theorem~\ref{intro:theorem-B} shows that this obstruction
disappears in the two-sided case. As a consequence, solvability is
closed under extensions.

\begin{corollary*}\label{intro:extension-solvability}
Let \(B\) be a two-sided skew brace. Then \(B\) is solvable if and only
if there exists an ideal \(I\) of \(B\) such that both \(I\) and \(B/I\)
are solvable.
\end{corollary*}

Combining this result with the structural theory of finite two-sided
skew braces developed in \cite{TrappeniersTwoSided}, we obtain the
following equivalence.

\begin{corollary*}\label{intro:group-solvability}
Let \(B\) be a finite two-sided skew brace. Then the following
conditions are equivalent:
\begin{enumerate}[label=\textup{(\roman*)}]
    \item \(B\) is solvable;
    \item \((B,+)\) is solvable;
    \item \((B,\cdot)\) is solvable.
\end{enumerate}
\end{corollary*}

As further consequences, we obtain the following results.

\begin{corollary*}\label{intro:abelian-multiplicative-solvable}
Let \(B\) be a finite skew brace such that \((B,\cdot)\) is abelian.
Then \(B\) is solvable.
\end{corollary*}

\begin{corollary*}\label{intro:odd-order-two-sided}
Let \(B\) be a finite two-sided skew brace of odd order. Then \(B\) is
solvable.
\end{corollary*}

\begin{corollary*}\label{intro:two-prime-order-two-sided}
Let \(B\) be a finite two-sided skew brace such that $|B|=p^nq^m$
for some primes \(p\) and \(q\) and some non-negative integers \(n\)
and \(m\). Then \(B\) is solvable.
\end{corollary*}

The equivalence between the solvability of a finite two-sided skew
brace and that of its associated groups relies essentially on the
finiteness assumption. Indeed, Nasybullov \cite[Section~3]{Nasybullov2019} constructed infinite
two-sided skew braces whose additive group is abelian while the
multiplicative group is not solvable. Thus, in the infinite setting,
solvability of the additive group does not in general imply solvability
of the multiplicative group. Nevertheless, a weaker residual form of
this phenomenon still survives: although the multiplicative group
itself may fail to be solvable, all of its finite homomorphic images
remain solvable.

\begin{athm}\label{intro:theorem-C}
Let \(B=(B,+,\cdot)\) be a two-sided skew brace such that \((B,+)\) is
solvable. Then every finite homomorphic image of \((B,\cdot)\) is
solvable.
\end{athm}

The paper is organized as follows. In Section~\ref{sec: Preliminaries}, we recall the basic
definitions and preliminary results concerning skew braces, ideals,
commutators and solvability. In Section~\ref{sec: Theorem A}, we prove
Theorem~\ref{intro:theorem-A} and derive its consequences for simple
skew braces and ideals of prime index. We also provide an example
showing that nilpotency of the multiplicative group does not imply left
nilpotency. In Section~\ref{sec: Solvability two sided}, we study the solvability of
two-sided skew braces, both in the finite and in the general setting. We
prove Theorems~\ref{intro:theorem-B} and~\ref{intro:theorem-C}, establish
the extension property for solvability, and, in the finite case, derive
the equivalence between the solvability of a two-sided skew brace and
that of its additive and multiplicative groups.

\section{Preliminaries} \label{sec: Preliminaries}

We recall some basic definitions and properties of skew braces that
will be used throughout the paper. Let $(B,+,\cdot)$ be a skew brace as in the introduction.
The identity elements of $(B,+)$ and $(B,\cdot)$ coincide and will be denoted by
\(0\). For every \(a\in B\), define $\lambda_a:B\longrightarrow B$ by 
$\lambda_a(b)=-a+a\cdot b.$
The skew brace identity implies that $\lambda_a\in\operatorname{Aut}(B,+)$
for every \(a\in B\), and that the map $\lambda:(B,\cdot)\longrightarrow\operatorname{Aut}(B,+)$ given by $\lambda(a)=\lambda_{a}$
is a group homomorphism. Moreover, the lambda map allows one to express each of the two group
operations in terms of the other. Indeed, for all \(a,b\in B\),
\[
a\cdot b=a+\lambda_a(b)
\qquad\text{and}\qquad
a+b=a\cdot\lambda_a^{-1}(b).
\]
It is also convenient to introduce the operation
$a*b=\lambda_a(b)-b$
for all \(a,b\in B\). If \(H\) and \(K\) are subsets of \(B\), we denote by $H*K=\bigl\langle h*k \mid h\in H,\ k\in K\bigr\rangle_+$
the additive subgroup of \((B,+)\) generated by all elements \(h*k\), with
\(h\in H\) and \(k\in K\). We write $H^{2}$ for $H*H$. A skew brace \(B=(B,+,\cdot)\) is called \emph{trivial} if its two
operations coincide, that is, $a\cdot b=a+b$
for all \(a,b\in B\). Equivalently $B*B=0.$
For a group \(G\), we denote by \(\Triv(G)\) the trivial skew brace
whose additive and multiplicative groups are both equal to \(G\). A skew brace \(B\) is called \emph{abelian} if it is trivial and its
additive group is abelian. Equivalently, both associated groups are
abelian and the two operations coincide. The \emph{left series} of a skew brace \(B\) is defined recursively by
\[
B^1=B
\qquad\text{and}\qquad
B^{n+1}=B*B^n
\]
for every \(n\geq 1\). We say that \(B\) is \emph{left nilpotent} if
\(B^m=0\) for some positive integer \(m\).
A subset \(C\subseteq B\) is called a \emph{subbrace} of \(B\) if $(C,+)\leq(B,+)$ and $(C,\cdot)\leq(B,\cdot).$ 
A subgroup \(I\leq(B,+)\) is called a \emph{left ideal} of \(B\) if it
is invariant under the lambda action, that is,
$\lambda_b(I)=I$
for every \(b\in B\). We record the following useful observation.

\begin{Proposition}\label{lem:characteristic-left-ideal}
Let \(B\) be a skew brace and let
\(H\) be a left ideal of \(B\). Then $H\leq(B,\cdot).$
\end{Proposition}

\begin{proof}
If \(h,k\in H\), then $h\cdot k=h+\lambda_h(k)\in H$ and  $h_{\cdot}^{-1}=\lambda_h^{-1}(-h)\in H$ and thus \(H\leq(B,\cdot)\).
\end{proof}

A subset \(I\subseteq B\) is called an \emph{ideal} of \(B\), and we
write \(I\trianglelefteq B\), if \(I\) is a left ideal which is normal
in both the additive and multiplicative groups. We shall also use the
well-known fact that \(B*B\) is an ideal of \(B\).

\begin{Proposition}\label{lem:left-ideal-criterion} Let \(B\) be a skew brace and let \(I\) be a left ideal of \(B\) such that \(I\trianglelefteq(B,+)\). Then \(I\) is an ideal of \(B\) if and only if $I*B\leq I.$
\end{Proposition}

\begin{proof}
Suppose first that \(I\) is an ideal of \(B\). Let \(i\in I\) and
\(b\in B\). Since \(I\trianglelefteq(B,\cdot)\), we have
\(i\cdot b\in b\cdot I\). Moreover, since \(I\) is a left ideal we have $b\cdot I=b+\lambda_b(I)=b+I=I+b,$
where the last equality follows from \(I\trianglelefteq(B,+)\).
Hence $i+\lambda_i(b)=i\cdot b\in I+b.$
Since \(i\in I\), it follows that \(\lambda_i(b)\in I+b\), and therefore $i*b=\lambda_i(b)-b\in I.$
Thus \(I*B\leq I\). Conversely, suppose that \(I*B\leq I\). For \(i\in I\) and \(b\in B\),
we have \(\lambda_i(b)=i*b+b\in I+b\), and hence
\[
i\cdot b=i+\lambda_i(b)\in I+b=b+I=b\cdot I.
\]
Therefore \(I\cdot b\subseteq b\cdot I\). Since both sets have
cardinality \(|I|\), we obtain \(I\cdot b=b\cdot I\) for every
\(b\in B\). Thus \(I\trianglelefteq(B,\cdot)\), and hence \(I\) is an
ideal of \(B\).
\end{proof}

If \(I\) is an ideal of \(B\), then the quotient set \(B/I\) inherits
a skew brace structure defined by
\[
(a+I)+(b+I)=a+b+I
\qquad\text{and}\qquad
(a+I)\cdot(b+I)=a\cdot b+I.
\]
A skew brace \(B\) is called \emph{simple} if $|B|>1$ and its only ideals are \(0\) and \(B\).  Let \(I,J\) be ideals of \(B\). Recall that the commutator
\([I,J]_B\) is the ideal of \(B\) generated by
\[
[I,J]_+\cup [I,J]_\cdot
\cup
\{\,i\cdot j-(i+j)\mid i\in I,\ j\in J\,\}.
\]

For an ideal \(I\trianglelefteq B\), define its derived series with
respect to \(B\) by
\[
\partial_B^0(I)=I,
\qquad
\partial_B^{n+1}(I)
=
[\partial_B^n(I),\partial_B^n(I)]_B
\]
for every \(n\geq 0\). We say that \(I\) is \emph{\(B\)-solvable}, or
\emph{solvable with respect to \(B\)}, if
$\partial_B^n(I)=0$
for some \(n\geq 0\). When \(I=B\), we simply write $\partial^n(B)=\partial_B^n(B).$
In particular, $\partial(B)=\partial^1(B)=[B,B]_B.$
We say that \(B\) is \emph{solvable} if $\partial^n(B)=0$
for some \(n\geq 0\).

\section{Proof of Theorem \ref{intro:theorem-A}} \label{sec: Theorem A}

We begin this section with the following
group-theoretic fixed-point lemma, which will be used to control the
action induced by \(p\)-elements on a suitable quotient of the additive
group.

\begin{lemma}\label{lem:fixed-point}
Let \(G\) be a finite solvable group such that \(O_p(G)=1\), and let
\(\alpha\in\Aut(G)\) have \(p\)-power order. Suppose that \(\alpha\)
fixes elementwise a Hall \(p'\)-subgroup of \(G\). Then \(\alpha=1\).
\end{lemma}

\begin{proof}
Let \(F=F(G)\), and consider the semidirect product $\Gamma=G\rtimes\langle\alpha\rangle.$
Since \(F\) is nilpotent, its Sylow \(p\)-subgroup is characteristic in
\(F\), and hence normal in \(G\). Since \(O_p(G)=1\), it follows that
\(F\) is a \(p'\)-group. Let \(H\) be a Hall \(p'\)-subgroup of \(G\) fixed elementwise by
\(\alpha\). Since \(F\trianglelefteq G\) is a \(p'\)-group, we have
\(F\leq H\). Therefore $[F,\alpha]=1.$
Moreover, since \(F\trianglelefteq G\), we have \([G,F]\leq F\), and
hence $[G,F,\alpha]\leq [F,\alpha]=1.$ We also have $[F,\alpha,G]=1.$
By the Three Subgroups Lemma
\cite[Corollary~4.10]{Isaacs}, it follows that
$[\alpha,G,F]=1.$
Thus $[\alpha,G]\leq C_G(F).$
Since \(G\) is solvable, its Fitting subgroup is self-centralizing; more
precisely, $C_G(F)\leq F;$
see \cite[Problem~3B.14]{Isaacs}. Consequently,
$[G,\alpha]\leq F.$
As \([F,\alpha]=1\), we obtain $[G,\alpha,\alpha]=1.$
For every \(g\in G\) and every \(n\geq 0\), the standard commutator
identity gives
\[
[g,\alpha^{n+1}]
=
[g,\alpha^n]\,[g,\alpha]^{\alpha^n}.
\]
Since \([g,\alpha,\alpha]=1\), the element \([g,\alpha]\) is fixed by
\(\alpha\). Hence
\[
[g,\alpha]^{\alpha^i}=[g,\alpha]
\]
for every \(i\geq 0\), and therefore
\[
[g,\alpha^n]
=
[g,\alpha]\,[g,\alpha]^\alpha\cdots
[g,\alpha]^{\alpha^{n-1}}
=
[g,\alpha]^n.
\]

Let \(|\alpha|=p^m\). Then, for every \(g\in G\),
\[
1=[g,\alpha^{p^m}]=[g,\alpha]^{p^m}.
\]
On the other hand,
\[
[g,\alpha]\in [G,\alpha]\leq F,
\]
and \(F\) is a \(p'\)-group. It follows that \([g,\alpha]\) has both
\(p\)-power order and order coprime to \(p\), and hence $[g,\alpha]=1.$
Thus \([G,\alpha]=1\), so \(\alpha\) fixes every element of \(G\).
Therefore \(\alpha=1\).
\end{proof}

We are now ready to prove Theorem~\ref{intro:theorem-A}.

\begin{proof}[Proof of Theorem \ref{intro:theorem-A}]
Since \((B,\cdot)\) is nilpotent, the additive group \((B,+)\) is
solvable by \cite[Theorem~1.3(c)]{TsangQinSolvability}. Therefore $F(B,+) \ne 1$. Let \(p\) be a prime divisor of \(|F(B,+)|\), and set $Q=O_p(B,+).$
We prove that \(Q\) is an ideal of \(B\).
Since \(Q\) is characteristic in \((B,+)\), we have $\lambda_a(Q)=Q$
for every \(a\in B\). 
Thus, by Proposition \ref{lem:characteristic-left-ideal}, \(Q\) is also a \(p\)-subgroup of \((B,\cdot)\). We are left to prove that $Q$ is normal in $(B,\cdot)$. Let \(P\in\operatorname{Syl}_p(B,\cdot)\). Since \((B,\cdot)\) is
nilpotent, we have $(B,\cdot)=P\times H,$
where $H=O_{p'}(B,\cdot)$
is the unique Hall \(p'\)-subgroup of \((B,\cdot)\). Since \(Q\) is a
multiplicative \(p\)-subgroup, we have $Q\leq P.$  Since both \((B,+)\) and \((B,\cdot)\) are solvable, the Hall theorem for
finite skew braces \cite[Theorem~2.8]{TrumanSylow} ensures that $H=O_{p'}(B,\cdot)$
is the multiplicative group of a Hall \(p'\)-subbrace of \(B\). In
particular, \((H,+)\) is a Hall \(p'\)-subgroup of \((B,+)\). Let \(x\in Q\) and \(h\in H\). Since \(x\in P\) and $(B,\cdot)=P\times H,$ we have $x\cdot h=h\cdot x$. Thus we obtain

\[
\begin{aligned}
\lambda_x(h)-h
   &=(-x+x\cdot h)-h\\
   &=(-x+h\cdot x)-h\\
   &=-x+(h\cdot x-h)\\
   &=-x+\bigl(h+\lambda_h(x)-h\bigr).
\end{aligned}
\]
Since \(x,\lambda_h(x)\in Q\) and \(Q\trianglelefteq(B,+)\), we have $h+\lambda_h(x)-h\in Q,$
and hence \(\lambda_x(h)-h\in Q\).
Consider now the quotient group $(B,+)/Q.$
Since \(Q=O_p(B,+)\), we have $O_p\bigl((B,+)/Q\bigr)=1.$
For each \(x\in Q\), define
\[
\overline{\lambda_x}\colon (B,+)/Q\longrightarrow (B,+)/Q,
\qquad
\overline{\lambda_x}(b+Q)=\lambda_x(b)+Q.
\]
Since \(Q\) is invariant under \(\lambda_x\), this map is well defined and
belongs to \(\operatorname{Aut}((B,+)/Q)\).
The automorphism \(\overline{\lambda_x}\) has \(p\)-power order. Indeed,
since \(Q\) is a \(p\)-subgroup of \((B,\cdot)\), the element \(x\) has
\(p\)-power order in \((B,\cdot)\). As
$\lambda\colon (B,\cdot)\longrightarrow \operatorname{Aut}(B,+)$
is a group homomorphism, the order of \(\lambda_x\) divides the order of
\(x\), and is therefore a power of \(p\). Finally,
\(\overline{\lambda_x}\) is the automorphism induced by \(\lambda_x\) on
\((B,+)/Q\), so its order divides the order of \(\lambda_x\). Hence
\(\overline{\lambda_x}\) also has \(p\)-power order. Observe that $(H+Q)/Q$
is a Hall \(p'\)-subgroup of \((B,+)/Q\)
, and the previous calculation
shows that
\[
\overline{\lambda_x}(h+Q)=h+Q
\]
for every \(h\in H\). Thus \(\overline{\lambda_x}\) fixes a Hall
\(p'\)-subgroup of \((B,+)/Q\) elementwise. By
Lemma~\ref{lem:fixed-point} we conclude that $\overline{\lambda_x}=1.$
Therefore $\lambda_x(b)-b\in Q$
for every \(x\in Q\) and \(b\in B\). Equivalently $Q*B\leq Q.$ By Proposition \ref{lem:left-ideal-criterion} we conclude that $Q$ is an ideal of $B$.
Since
\[
F(B,+)=\prod_{q\mid |F(B,+)|} O_q(B,+)
\]
is the direct product of the ideals \(O_q(B,+)\), it follows that
\(F(B,+)\) is a non-zero ideal of \(B\).
\end{proof}

\begin{corollary} \label{cor: simple-multiplicative-nilpotent}
    Let $B$ be a finite simple skew brace such that \((B,\cdot)\) is
nilpotent. Thus $B \simeq  \Triv(C_p)$.
\end{corollary}

\begin{proof}
By Theorem~\ref{intro:theorem-A}, the Fitting subgroup \(F(B,+)\) is
a non-zero ideal of \(B\). As \(B\) is
simple, it follows that $F(B,+)=B.$
Hence \((B,+)\) is nilpotent.
Let \(p\) be a prime divisor of \(|B|\), and let
\(P\in\operatorname{Syl}_p(B,+)\). Since \((B,+)\) is nilpotent, \(P\)
is the unique Sylow \(p\)-subgroup of \((B,+)\), and hence it is
characteristic in \((B,+)\). Therefore \(P\) is invariant under the
lambda action, so \(P\) is a left ideal of \(B\). In particular, by Proposition \ref{lem:characteristic-left-ideal}, \(P\)
is also a multiplicative subgroup of \(B\). Since \(|P|\) is the full
\(p\)-part of \(|B|\), it follows that $P\in\operatorname{Syl}_p(B,\cdot).$
As \((B,\cdot)\) is nilpotent, \(P\trianglelefteq(B,\cdot)\). Moreover,
\(P\trianglelefteq(B,+)\), since \(P\) is characteristic in \((B,+)\).
Thus \(P\) is an ideal of \(B\). By the simplicity of \(B\), we obtain
\(P=B\), and consequently \(|B|\) is a power of \(p\). Thus \(B\) is a finite simple skew brace of \(p\)-power order. By
\cite[Proposition~4.4]{CedoSmoktunowiczVendramin2019}, every finite skew
brace of \(p\)-power order is left nilpotent. Hence the descending
left series of \(B\) terminates at zero. Since \(B\) is simple, the
ideal \(B*B\) is either \(0\) or \(B\). The latter possibility is
excluded by left nilpotency, and therefore $B*B=0.$
Thus \(B\) is a trivial skew brace. Since \(B\) is simple, its additive
group is a simple \(p\)-group, and hence isomorphic to \(C_p\).
Consequently $B\simeq \Triv(C_p).$
\end{proof}

\begin{corollary} \label{ideal index p}
Let \(B\) be a finite skew brace such that \((B,\cdot)\) is nilpotent.
Then \(B\) admits an ideal of index \(p\) for some prime \(p\).
\end{corollary}

\begin{proof}
Let \(I\) be a maximal proper ideal of \(B\). Then \(B/I\) is a finite
simple skew brace. Moreover, its multiplicative group $(B/I,\cdot)\cong (B,\cdot)/(I,\cdot)$
is nilpotent, being a quotient of the nilpotent group \((B,\cdot)\).
Hence, by Corollary~\ref{cor: simple-multiplicative-nilpotent}, there
exists a prime \(p\) such that $B/I\cong \Triv(C_p).$
Therefore $[B:I]=|B/I|=p,$
and thus \(I\) is an ideal of index \(p\) in \(B\).
\end{proof}

\begin{corollary}
Let \(B\) be a finite skew brace such that \((B,\cdot)\) is nilpotent.
Then $B*B\neq B$ and $\partial(B)\neq B.$
\end{corollary}

\begin{proof}
By Corollary~\ref{ideal index p}, \(B\) admits an ideal \(I\) of index
\(p\) for some prime \(p\). Hence \(B/I\) is a skew brace of order \(p\),
and therefore $B/I\simeq \Triv(C_p).$
In particular, the star product on \(B/I\) is trivial. Thus, for every
\(a,b\in B\),
\[
(a+I)*(b+I)=a*b+I=I,
\]
and consequently $B*B\leq I.$
Since \(I\neq B\), it follows that \(B*B\neq B\). Moreover, \(B/I\) is abelian, and hence
$\partial(B/I)=0.$
Since the canonical projection \(B\to B/I\) maps \(\partial(B)\) onto
\(\partial(B/I)\), we obtain $\partial(B)\leq I.$
Since \(I\) is a proper ideal of \(B\), it follows that $\partial(B)\neq B.$
\end{proof}

\begin{example} \rm \label{example:nilpotent-not-left-nilpotent}
The preceding result does not imply either that \(B\) is left
nilpotent or that \(B\) admits an ideal of index \(p\) for every prime
\(p\) dividing \(|B|\). Indeed, write
$S_3=\langle r,s\mid r^3=s^2=1,\ srs=r^{-1}\rangle.$
The group \(S_3\) admits the exact factorization $S_3=\langle r\rangle\langle s\rangle
   \simeq C_3\rtimes C_2.$
Hence, by \cite[Theorem~2.3]{SmoktunowiczVendramin2018}, there exists a
skew brace \(B\) whose additive group is isomorphic to \(S_3\) and whose
multiplicative group is isomorphic to $C_3\times C_2\simeq C_6.$
In particular, \((B,\cdot)\) is nilpotent.
More explicitly, identifying the underlying set of \(B\) with \(S_3\),
every element can be written uniquely as \(r^is^j\), where
\(0\leq i\leq 2\) and \(0\leq j\leq 1\), and the multiplicative
operation is given by $(r^is^j)\cdot b=r^i b s^j$
for every \(b\in S_3\). It follows that
\[
\lambda_{r^is^j}(b)
=(r^is^j)^{-1}r^i b s^j
=s^jbs^j.
\]
Thus \(\lambda_{r^i}=\operatorname{id}\), whereas
\(\lambda_{r^is}\) is conjugation by \(s\).
We now compute the left series. Since the quotient
\(S_3/\langle r\rangle\) is fixed by every lambda map, we have $B*B\leq \langle r\rangle.$
On the other hand,
\[
s*r
=\lambda_s(r)r^{-1}
=srsr^{-1}
=r^{-1}r^{-1}
=r,
\]
and hence $B*B=\langle r\rangle.$
Moreover, the same computation gives $B*\langle r\rangle=\langle r\rangle.$
Consequently, if \(B^1=B\) and \(B^{n+1}=B*B^n\), then $B^n=\langle r\rangle$
for every \(n\geq 2\). Therefore \(B\) is not left nilpotent. Finally, \(B\) does not admit an ideal of index \(3\). Indeed, such an
ideal would be an additive normal subgroup of order \(2\) in
\((B,+)\simeq S_3\), but \(S_3\) has no normal subgroup of order \(2\).
\end{example}

\section{Solvability of two-sided skew braces} \label{sec: Solvability two sided}

\subsection{The finite case}

In the previous section, we studied finite skew braces with nilpotent
multiplicative group. In particular, we proved that such a skew brace
cannot be simple, apart from the trivial skew braces of prime order.
It is therefore natural to ask whether the nilpotency of the
multiplicative group also forces the skew brace itself to be solvable. This is not the case, as the next example shows. 

\begin{example} \rm
In
\cite[Example~38]{BallesterBolinchesEstebanRomeroJimenezSeralPerezCalabuig2024}
the authors construct the skew brace $B=\operatorname{SmallBrace}(32,24003),$
which is not solvable. Since \(B\) has order \(32\), its multiplicative
group is a \(2\)-group and hence is nilpotent.
\end{example}

Thus the nilpotency
of \((B,\cdot)\) alone is too weak to guarantee the solvability of
\(B\). This leads to the stronger question of whether a finite skew brace with
abelian multiplicative group must be solvable. To address this question,
we work in the more general setting of two-sided skew braces. Recall that a skew brace \(B=(B,+,\cdot)\) is called \emph{two-sided} if,
in addition to the usual skew brace identity
it also satisfies
\[
(a+b)\cdot c=a\cdot c-c+b\cdot c
\]
for all \(a,b,c\in B\).

\begin{remark} \label{rem:abelian-multiplicative-two-sided}
Every skew brace with abelian multiplicative group is two-sided. Indeed,
for all \(a,b,c\in B\), we have
\[
(a+b)\cdot c
=c\cdot(a+b)
=c\cdot a-c+c\cdot b
=a\cdot c-c+b\cdot c.
\]
\end{remark}

We recall the following fundamental property of two-sided skew braces.

\begin{Proposition}\label{prop:characteristic-additive-ideal}
Let \(B\) be a two-sided skew brace, and let \(I\) be a characteristic
subgroup of \((B,+)\). Then \(I\) is an ideal of \(B\).
\end{Proposition}

\begin{proof}
This follows from \cite[Proposition~2.3]{TrappeniersTwoSided}.
\end{proof}

\begin{lemma}\label{commutatorTwosided}
Let \(B\) be a two-sided skew brace. Then
\[
\partial(B)=[B,B]_+ + B^2.
\]
\end{lemma}

\begin{proof}
Since \(B\) is two-sided, the additive commutator subgroup
\([B,B]_+\) is characteristic in \((B,+)\), and hence it is an ideal of
\(B\) by Proposition~\ref{prop:characteristic-additive-ideal}.
Moreover, \(B^2=B*B\) is an ideal of \(B\). Therefore $[B,B]_+ + B^2$
is an ideal of \(B\). By
\cite[Theorem~3.6]{BallesterBolinchesEtAlCentral}, we have $[B,B]_B=[B,B]_+ + B^2.$
Since \(\partial(B)=[B,B]_B\), the result follows.
\end{proof}

\begin{lemma}\label{lem:star-product-identity}
Let \(B\) be a skew brace. Then, for all \(x,y,b\in B\),
\[
(x\cdot y)*b=x*(y*b)+y*b+x*b.
\]
\end{lemma}

\begin{proof}
Since \(\lambda\colon(B,\cdot)\to\operatorname{Aut}(B,+)\) is a group
homomorphism, \(\lambda_{x\cdot y}=\lambda_x\lambda_y\). Hence
\[
\begin{aligned}
(x\cdot y)*b
&=\lambda_{x\cdot y}(b)-b
 =\lambda_x\bigl(\lambda_y(b)\bigr)-b\\
&=\lambda_x(y*b+b)-b
 =\lambda_x(y*b)+\lambda_x(b)-b\\
&=x*(y*b)+y*b+x*b.
\end{aligned}
\]
\end{proof}
We are now ready to prove Theorem~\ref{intro:theorem-B}.
\begin{proof}[Proof of Theorem \ref{intro:theorem-B}]

Let \(K=[I,I]_I\). Since \(I\) is itself a two-sided skew brace,
Lemma~\ref{commutatorTwosided} gives $K=[I,I]_+ + I^2.$
By \cite[Lemmas~5.2 and~5.3]{TrappeniersTwoSided}, both \([I,I]_+\) and
\(I^2\) are left ideals of \(B\) and are normal in \((B,\cdot)\).
Since the sum of two left ideals coincides with their multiplicative
product, we have $K=[I,I]_+ + I^2=[I,I]_+\cdot I^2.$
Therefore \(K\) is a left ideal of \(B\) and is normal in
\((B,\cdot)\).
It remains to prove that
\((K,+)\trianglelefteq (B,+)\). We first prove that \(K*B\subseteq K\). Fix \(b\in B\) and set
\[
L_b=\{z\in I \mid z*b\in K\}.
\]
Consider the map \(\varphi_b\colon I\to I/K\) given by $\varphi_b(z)=z*b+K.$
This is well-defined because \(I*B\subseteq I\).
We claim that \(\varphi_b\) is additive. 
Indeed, for \(z,w\in I\), the right distributive law yields
\[
\begin{aligned}
(z+w)*b
   &=-(z+w)+(z+w)\cdot b-b\\
   &=-w-z+\bigl(z\cdot b-b+w\cdot b\bigr)-b\\
   &=-w+\bigl(-z+z\cdot b-b\bigr)+w+\bigl(-w+w\cdot b-b\bigr)\\
   &=-w+z*b+w+w*b.
\end{aligned}
\]
Since \(I/K\) is an abelian skew brace, the additive group \((I/K,+)\)
is abelian. Therefore, using the previous identity, we obtain
\[
\begin{aligned}
\varphi_b(z+w)
   &=(z+w)*b+K\\
   &=-w+z*b+w+w*b+K\\
   &=z*b+w*b+K\\
   &=(z*b+K)+(w*b+K)\\
   &=\varphi_b(z)+\varphi_b(w).
\end{aligned}
\]
Thus \(\varphi_b\) is additive.
We now prove that \(\varphi_b\) is multiplicative. Let \(x,y\in I\).
Using Lemma \ref{lem:star-product-identity} we obtain
\[
\begin{aligned}
\varphi_b(x\cdot y)
   &=(x\cdot y)*b+K\\
   &=x*(y*b)+y*b+x*b+K.
\end{aligned}
\]
Since \(I\) is an ideal of \(B\), by Proposition \ref{lem:left-ideal-criterion},  we have \(y*b\in I\). Therefore $x*(y*b)\in I*I\leq K,$
and hence
\[
\begin{aligned}
\varphi_b(x\cdot y)
   &=y*b+x*b+K\\
   &=x*b+y*b+K\\
   &=\varphi_b(x)+\varphi_b(y),
\end{aligned}
\]
where the second equality follows from the fact that \((I/K,+)\) is
abelian. Since \(I/K\) is a trivial skew brace, its additive and
multiplicative operations coincide. Consequently, $\varphi_b(x\cdot y)
   =\varphi_b(x)\cdot\varphi_b(y).$
Thus \(\varphi_b\) is multiplicative. Therefore \(L_b=\ker\varphi_b\) is an ideal of \(I\).
Since \(\varphi_b\colon I\to I/K\) is a skew brace homomorphism and
\(I/K\) is a trivial abelian skew brace, we have, for all \(x,y\in I\),
\[
\varphi_b([x,y]_+)=0,\qquad
\varphi_b([x,y]_\cdot)=0,
\qquad\text{and}\qquad
\varphi_b(x*y)
=\varphi_b(x)*\varphi_b(y)=0.
\]
Thus \(\ker\varphi_b=L_b\) contains the additive commutators, the
multiplicative commutators, and \(I*I\). Since \(K=[I,I]_I\) is the
ideal of \(I\) generated by these elements, it follows that $K\leq L_b.$ Since \(b\in B\) was arbitrary, we conclude that $K*B\leq K.$
We now prove that \(K\) is normal in \((B,+)\). Fix \(b\in B\). Since
\(K\) is a left ideal, we have \(\lambda_b(K)=K\), and therefore
\[
b\cdot K=b+\lambda_b(K)=b+K.
\]
On the other hand, since \(K*B\leq K\), for every \(k\in K\) we have
\(k*b\in K\). Hence $\lambda_k(b)=k*b+b\in K+b,$
and consequently
\[
k\cdot b=k+\lambda_k(b)\in K+b.
\]
Thus $K\cdot b\subseteq K+b.$
Since \(K\trianglelefteq(B,\cdot)\), we have \(b\cdot K=K\cdot b\), and
therefore
\[
b+K=b\cdot K=K\cdot b\subseteq K+b.
\]
Adding \(-b\) on the right gives $b+K-b\subseteq K.$ Applying the same argument with \(-b\) in place of \(b\), we obtain $-b+K\subseteq K-b.$
Adding \(b\) on the left and on the right yields $K\subseteq b+K-b.$ Hence $b+K-b=K.$
Since \(b\in B\) was arbitrary, it follows that
$K\trianglelefteq(B,+).$
Together with the facts that \(K\) is a left ideal of \(B\) and
\(K\trianglelefteq(B,\cdot)\), this shows that \(K\trianglelefteq B\).

\end{proof}

\begin{lemma}\label{lem:commutators-independent-ambient}
Let \(B\) be a two-sided skew brace, and let \(I,J\trianglelefteq B\)
with \(J\leq I\). Then $[J,J]_B=[J,J]_I=[J,J]_J.$
\end{lemma}

\begin{proof}
By Theorem~\ref{intro:theorem-B}, the internal commutator
\([J,J]_J\) is an ideal of \(B\). Since it contains all the generators
of \([J,J]_B\), the minimality of \([J,J]_B\) gives $[J,J]_B\leq [J,J]_J.$ Conversely, since \(J\trianglelefteq B\), the ideal \([J,J]_B\) is
contained in \(J\). Hence \([J,J]_B\) is an ideal of the skew brace
\(J\) containing all the generators of \([J,J]_J\). Therefore $[J,J]_J\leq [J,J]_B.$
Thus $[J,J]_B=[J,J]_J.$ Since \(J\trianglelefteq I\), the same argument applied inside \(I\)
gives $[J,J]_I=[J,J]_J.$
Consequently, $[J,J]_B=[J,J]_I=[J,J]_J.$
\end{proof}

\begin{corollary}\label{criteria-for-solvability}
Let \(B\) be a two-sided skew brace. Then \(B\) is solvable if and only
if there exists an ideal \(I\trianglelefteq B\) such that \(I\), regarded
as a skew brace in its own right, and \(B/I\) are both solvable.
\end{corollary}

\begin{proof}
If \(B\) is solvable, it is enough to take \(I=B\). Conversely, suppose that \(I\) and \(B/I\) are solvable. Let $\pi\colon B\longrightarrow B/I$
be the canonical projection. Since skew brace homomorphisms preserve
commutators, we have $\pi\bigl(\partial^n(B)\bigr)=\partial^n(B/I)$
for every \(n\geq 0\). Since \(B/I\) is solvable, there exists
\(m\geq 0\) such that $\partial^m(B/I)=0.$
Therefore $\partial^m(B)\leq I.$ Let
\[
I^{(0)}=I,
\qquad
I^{(n+1)}=[I^{(n)},I^{(n)}]_I
\]
for every \(n\geq 0\). Thus \(\{I^{(n)}\}_{n\geq 0}\) is the derived
series of \(I\), regarded as a skew brace in its own right. We claim that \(I^{(n)}\trianglelefteq B\) for every \(n\geq 0\).
This is clear for \(n=0\). Suppose that
\(I^{(n)}\trianglelefteq B\). Since \(I^{(n)}\leq I\), by
Lemma~\ref{lem:commutators-independent-ambient} we have
\[
[I^{(n)},I^{(n)}]_B
=
[I^{(n)},I^{(n)}]_I
=
[I^{(n)},I^{(n)}]_{I^{(n)}}.
\]
By Theorem~\ref{intro:theorem-B}, the last term is an ideal of \(B\).
Hence $I^{(n+1)}
=
[I^{(n)},I^{(n)}]_I
\trianglelefteq B.$
The claim follows by induction. We now prove by induction on \(n\) that $\partial^{m+n}(B)\leq I^{(n)}$
for every \(n\geq 0\). For \(n=0\), this is precisely
\(\partial^m(B)\leq I\). Suppose that the assertion holds for some
\(n\geq 0\). Then
\[
\begin{aligned}
\partial^{m+n+1}(B)
&=
[\partial^{m+n}(B),\partial^{m+n}(B)]_B\\
&\leq
[I^{(n)},I^{(n)}]_B\\
&=
[I^{(n)},I^{(n)}]_I\\
&=
I^{(n+1)},
\end{aligned}
\]
where the equality in the third line follows from
Lemma~\ref{lem:commutators-independent-ambient}. Thus the assertion
follows by induction. Since \(I\) is solvable, there exists \(r\geq 0\) such that $I^{(r)}=0.$
Consequently, $\partial^{m+r}(B)\leq I^{(r)}=0,$
and therefore \(B\) is solvable.
\end{proof}

The preceding corollary does not hold for arbitrary skew braces; in
particular, the two-sidedness assumption is essential. Recall that a skew brace \(B\) is said to be \emph{weakly solvable} if the series
\[
B^{[1]}=B,
\qquad
B^{[n+1]}=B^{[n]}*B^{[n]}
\]
trivializes for some \(n_0\geq 1\), and $B^{[n]}/B^{[n+1]}$
is abelian for every \(n\geq 0\).

\begin{example} \rm
Consider the skew brace $B=\operatorname{SmallBrace}(32,24003)$
described in
\cite[Example~38]{%
BallesterBolinchesEstebanRomeroJimenezSeralPerezCalabuig2024}.
As shown in
\cite[Example~38]{BallesterBolinchesEstebanRomeroJimenezSeralPerezCalabuig2024},
this skew brace is weakly solvable but not solvable. Moreover, it has a
unique non-zero proper ideal \(I\), of order \(16\). Since weak solvability is inherited by subbraces, \(I\), regarded as a
skew brace in its own right, is weakly solvable. The computations in
\cite{BallesterBolinchesEstebanRomeroJimenezSeralPerezCalabuig2024}
show that every weakly solvable skew brace of order at most \(31\) is
solvable. Since \(|I|=16\), it follows that \(I\) is solvable as a skew
brace in its own right. Furthermore, \(I\) has index \(2\) in \(B\), and hence $B/I\simeq \Triv(C_2).$
Thus \(B/I\) is solvable, whereas \(B\) is not solvable. 
\end{example}

\begin{corollary}\label{cor:solvability-equivalences}
Let \(B\) be a finite two-sided skew brace. Then the following conditions
are equivalent:
\begin{enumerate}
    \item \(B\) is solvable;
    \item \((B,+)\) is solvable;
    \item \((B,\cdot)\) is solvable.
\end{enumerate}
\end{corollary}

\begin{proof} It is clear that if \(B\) is solvable, then both \((B,+)\) and \((B,\cdot)\) are solvable. Moreover, for a finite two-sided skew brace, \((B,+)\) is solvable if and only if \((B,\cdot)\) is solvable; see \cite[Theorem 4.20]{TrappeniersTwoSided}. It remains to prove that the solvability of \((B,+)\) implies the solvability of \(B\). Suppose, by contradiction, that this is false, and let \(B\) be a counterexample of minimal order. If \(B\) is simple, then, by the classification of finite simple two-sided skew braces \cite[Theorem~4.7]{TrappeniersTwoSided}, \(B\) is either trivial or almost trivial over a finite simple group. Since \((B,+)\) is solvable, this simple group must be cyclic of prime order. Hence $B\simeq \Triv(C_p),$ and therefore \(B\) is solvable, a contradiction. Thus \(B\) is not simple. Let \(I\) be a non-zero proper ideal of \(B\). Since \(I\) and \(B/I\) are two-sided skew braces and their additive groups are respectively a subgroup and a quotient of the solvable group \((B,+)\), both \((I,+)\) and \((B/I,+)\) are solvable. By the minimality of \(B\), the skew braces \(I\) and \(B/I\) are solvable. Hence Corollary~\ref{criteria-for-solvability} implies that \(B\) is solvable, again a contradiction. Therefore \(B\) is solvable. 

\end{proof}

\begin{remark}
In \cite[Corollary~4.20]{TrappeniersTwoSided}, Trappeniers proves the
equivalence between the solvability of the two associated groups of a
finite two-sided skew brace using a different notion of solvability,
namely the following one: a skew brace \(B\) is called solvable if $B^{[n]}=0$
for some \(n\geq 1\).
This notion is weaker than the one adopted in the present paper. Indeed, $B*B\leq \partial(B),$
and an induction gives $B^{[n]}\leq \partial^{\,n-1}(B)$
for every \(n\geq 1\). Consequently, every skew brace which is solvable
in our sense is also solvable in the sense of
\cite{TrappeniersTwoSided}. Therefore
Corollary~\ref{cor:solvability-equivalences} implies
\cite[Corollary~4.20]{TrappeniersTwoSided}.
\end{remark}

\begin{corollary}
    Let $B$ be a two sided skew brace. If 
\begin{enumerate}
    \item $|B|$ is odd, then $B$ is solvable.
    \item $|B|=p^{n}q^{m}$ with $p,q$ prime, then $B$ is solvable.
\end{enumerate}
\end{corollary}

\begin{proof}
Suppose first that \(|B|\) is odd. Then the additive group \((B,+)\) has odd order and is therefore solvable by the Feit--Thompson theorem (see the main Theorem of \cite{FeitThompson1963}). Hence \(B\) is solvable by Corollary~\ref{cor:solvability-equivalences}. Suppose now that $|B|=p^{n}q^{m}.$ By Burnside's \(p^{a}q^{b}\)-theorem \cite{Burnside1904}, the additive group \((B,+)\) is solvable. Once again, Corollary~\ref{cor:solvability-equivalences} implies that \(B\) is solvable.
\end{proof}

\begin{corollary}
    Let $B$ be a finite skew brace such that $(B,\cdot)$ is abelian. Thus $B$ is solvable.
\end{corollary}

\begin{proof}
Since \((B,\cdot)\) is abelian, \(B\) is two-sided by Remark~\ref{rem:abelian-multiplicative-two-sided}. Therefore Corollary~\ref{cor:solvability-equivalences} implies that \(B\) is solvable.
\end{proof}

\subsection{The infinite case and finite quotients}

The finiteness assumption in
Corollary~\ref{cor:solvability-equivalences} cannot be omitted.
Indeed, Nasybullov constructed infinite two-sided skew braces with
abelian additive group and non-solvable multiplicative group; see
\cite[Section~3]{Nasybullov2019}. Nevertheless, a residual form of the
finite result survives: whenever the additive group is solvable, every
finite homomorphic image of the multiplicative group is solvable. We first recall the relevant group-theoretic notion. A group \(G\) is
called an \emph{\(SN\)-group} if it admits an ascending, possibly
transfinite, normal series
\[
1=G_0\trianglelefteq G_1\trianglelefteq\cdots
\trianglelefteq G_\alpha=G
\]
whose factors are abelian. At every limit ordinal \(\lambda\), the
series is assumed to be continuous, that is,
\[
G_\lambda=\bigcup_{\beta<\lambda}G_\beta.
\]
Quotients of \(SN\)-groups are again \(SN\)-groups. Moreover, every
finite \(SN\)-group is solvable, since a transfinite ascending series
in a finite group contains only finitely many distinct terms.

\begin{lemma}\label{lem:multiplicative-SN}
Let \(B=(B,+,\cdot)\) be a two-sided skew brace such that \((B,+)\) is
abelian. Then \((B,\cdot)\) is an \(SN\)-group.
\end{lemma}

\begin{proof}
Since \((B,+)\) is abelian, \(B\) is a two-sided brace in the classical
sense. By \cite[Theorem~8.1.20]{CV}, there exists a Jacobson radical
ring structure on the additive group \((B,+)\) whose adjoint group is
isomorphic to \((B,\cdot)\). By \cite{ADS}, the adjoint group of a
Jacobson radical ring is an \(SN\)-group. Therefore \((B,\cdot)\) is
an \(SN\)-group.
\end{proof}

We now prove that the obstruction exhibited by Nasybullov cannot be
detected by finite quotients.

\begin{proof}[Proof of Theorem \ref{intro:theorem-C}]
Let \(n\) be the derived length of \((B,+)\). We argue by induction on
\(n\). Suppose first that \(n=1\). Then \((B,+)\) is abelian, and hence
\((B,\cdot)\) is an \(SN\)-group by
Lemma~\ref{lem:multiplicative-SN}. Since quotients of \(SN\)-groups are
again \(SN\)-groups and every finite \(SN\)-group is solvable, every
finite quotient of \((B,\cdot)\) is solvable.
Assume now that \(n>1\), and set $D=(B,+)'.$
Since \(D\) is characteristic in \((B,+)\), it is an ideal of \(B\) by
Proposition~\ref{prop:characteristic-additive-ideal}. In particular,
\(D\) is a two-sided skew brace and \(B/D\) is again a two-sided skew
brace. Let \(N\trianglelefteq(B,\cdot)\) have finite index, and let $\pi\colon(B,\cdot)\longrightarrow F=(B,\cdot)/N$
be the canonical epimorphism. Since \(D\) is an ideal of \(B\),
\((D,\cdot)\trianglelefteq(B,\cdot)\), and therefore $\pi(D)\cong (D,\cdot)/(D\cap N).$
The group \(\pi(D)\) is finite, being a subgroup of \(F\). Moreover,
the derived length of \((D,+)\) is strictly smaller than \(n\). By the
induction hypothesis, every finite quotient of \((D,\cdot)\) is
solvable. Hence \(\pi(D)\) is solvable. Furthermore, \(\pi(D)\trianglelefteq F\), and
$F/\pi(D)\cong (B,\cdot)/(DN).$
Since \(D\) is an ideal, we have $(B/D,\cdot)\cong (B,\cdot)/D,$
and consequently
\[
(B,\cdot)/(DN)
\cong
\bigl((B,\cdot)/D\bigr)\big/\bigl((DN)/D\bigr)
\cong
(B/D,\cdot)\big/\bigl((DN)/D\bigr).
\]
Thus \(F/\pi(D)\) is a finite quotient of \((B/D,\cdot)\). On the other hand, $(B/D,+)\cong (B,+)/D$
is abelian. By the base case, every finite quotient of
\((B/D,\cdot)\) is solvable. Therefore \(F/\pi(D)\) is solvable. We have shown that \(\pi(D)\) is a solvable normal subgroup of \(F\)
and that \(F/\pi(D)\) is solvable. Hence \(F\) is solvable. This
completes the induction.
\end{proof}

\end{document}